\documentclass[a4paper,twoside,11pt,reqno]{amsart}
\usepackage[utf8]{inputenc}
\usepackage[dvips]{graphicx}
\usepackage{amsmath}
\usepackage{amsfonts}
\usepackage{amssymb}
\usepackage{listings}
\usepackage{subfigure}
\usepackage{enumerate}
\usepackage{placeins}
\usepackage{color}             
\usepackage{colortbl}          
\usepackage{latexsym}
\usepackage{cite}
\usepackage{tikz}
\usepackage{pstricks}
\usepackage{pst-plot}
\usepackage{pst-slpe}
\usepackage{xcolor}

\definecolor{azulPersonal}{RGB}{0,128,255}

\def\R{\mathbb{R}}

\newcommand{\dis}{\displaystyle}
\newtheorem{Th}{Theorem}[section]
\newtheorem{Rem}{Remark}[section]
\newtheorem{Le}{Lemma}[section]

\newtheorem{Cor}{Corollary}[section]

\title[Attractors for the dissipative Signorini problem]{Global attractors for the Signorini problem  \\ with pointwise damping}

\author[J.E. Mu\~{n}oz Rivera]{Jaime E. Mu\~{n}oz Rivera}%
\address{Universidad del B\'\i o B\'\i o Collao 1202, Casilla 5C, \newline
\indent Concepci\'{o}n - Chile}

\address{Laboratório Nacional de Computação Científica (LNCC) , \newline
\indent Petr\'{o}polis - Brasil}
\email{jemunozrivera@gmail.com}%

\author[M.G. Naso]{Maria Grazia Naso}
\address{Dipartimento di Ingegneria Civile, Architettura, Territorio, Ambiente e di Matematica,
Universit\`{a} degli Studi di Brescia, Via Valotti 9, 25133
Brescia,
Italia.}%
\email{mariagrazia.naso@unibs.it}%

\begin{document}
\maketitle
\begin{center}
    \textit{Dedicated to Vittorino Pata on the occasion of his 60th birthday}
\end{center}

\vspace{1em}

\newcommand{\cuatex}[3]
{\begin{center}
\setlength{\unitlength}{2.54cm}
\begin{picture}(#1,#2)
\linethickness{2pt}
\put(0,0){\special{bmp:c:/suc_ensi/graduacao/calculoI/amarelo.bmp x=#1 in y=#2 in}}
\put(0,0){\framebox(#1,#2){\begin{minipage}{13.5cm}#3\end{minipage}}}
\end{picture}
\end{center}}


\begin{abstract}
The existence of global attractors is investigated for the Signorini problem with pointwise dissipation. It is shown that both the semilinear Signorini problem and the elastic obstacle problem with normal compliance exhibit exponential decay to zero and admit compact global attractors. To establish these results, the original problem is approximated by a hybrid PDE-ODE system, which allows for a rigorous analysis of well-posedness and the long-time behavior of its solutions.
\end{abstract}

\section{Introduction}

The existence of global attractors has been extensively studied in the context of semilinear partial differential equations (see, e.g., \cite{Lady,Teman,Pata-twoquestions,Miranville-Zelik,BV,Temam-EFN}). However, within the framework of variational inequalities, this subject remains largely unexplored. In this article, we investigate the existence of global attractors for the Signorini problem applied to the Timoshenko model.  

Let $0 < T \leq \infty$. We denote by $\varphi = \varphi(x, t) : (0, \ell) \times (0, T) \to \mathbb{R}$ the transverse displacement (vertical deflection) of the cross section at $x\in (0,\ell)$ and at time $t\in (0,T)$. Assuming that plane cross sections remain plane, the angle of rotation of a cross section is defined by $\psi = \psi(x, t) : (0, \ell) \times (0, T) \to \mathbb{R}$.
We then consider the system
 \begin{eqnarray}\label{SignoriniProblem}
 \begin{array}{lll}
  \rho_1\varphi_{tt}-k(\varphi_x+\psi)_x+\gamma_1\delta(x-\xi)\varphi_t +F(\varphi)=0,\\
  \noalign{\medskip}
  \rho_2\psi_{tt}-b\psi_{xx}+k(\varphi_x+\psi)+\gamma_2\delta(x-\xi)\psi_t+G(\psi)=0.
  \end{array}
 \end{eqnarray}
We suppose that, at $x=0$ and $x=\ell$,
 \begin{eqnarray}\label{BoundConditions1}
 \varphi(0,t)=0, \quad\psi_x(0,t)=0,\quad\psi (\ell,t)=0,
 \quad \text{in $(0,T)$}.
 \end{eqnarray}
  The joint at $x = \ell$ is modeled using the Signorini non-penetration condition (see, e.g., \cite{KS}). Specifically, the joint, characterized by a gap $g$, is asymmetrical, with $g = g_1 + g_2$, where $g_1 > 0$ and $g_2 > 0$ represent the upper and lower clearances, respectively, when the system is at rest. Consequently, the right end of the left beam is constrained to move vertically only between two stops
 \begin{eqnarray}\label{Cond.Signorini}
  g_1\leq\varphi(\ell,t)\leq g_2,\quad \text{with $0\leq t\leq T$}.
 \end{eqnarray}
 This condition assures that the displacement at $x=\ell$ is constrained between the stops
 $g_1$ and $g_2.$
 In addition,     the mathematical boundary conditions
 associated with this physical setup are as follows:
 \begin{equation}\label{Hip.Stress}
  \begin{array}{ccl}
   S(\ell,t)\geq0 &\mbox{if}& \varphi(\ell,t)=g_1,\\
 	&&\\
   S(\ell,t)=0 &\mbox{if}& g_1<\varphi(\ell,t)<g_2,\\
 	&&\\
   S(\ell,t)\leq 0 &\mbox{if}& \varphi(\ell,t)=g_2.\\
  \end{array}
 \end{equation}
  The initial conditions are prescribed as
 \begin{equation}\label{InitialConditions}
\begin{array}{lll}
  \varphi(x,0)=\varphi_0(x),& \varphi_t(x,0)=\varphi_1(x),  & \forall x\in(0,\ell),\\
  \noalign{\medskip}
  \psi(x,0)=\psi_0(x), & \psi_t(x,0)=\psi_1(x), &  \forall x\in(0,\ell),
  \end{array}
 \end{equation}
 where $\varphi_0, \varphi_1,\psi_0, \psi_1:(0,\ell)\to\R$ denote given functions.

The physical system is depicted in Fig.~\ref{FIG.1}, defined over the domain \((0, \ell) \times (0, T)\). The coefficients are established as follows: \(\rho_1 = \rho A\) represents the mass density, \(\rho_2 = \rho I\) denotes the mass moment of inertia, \(k = \kappa G A\) corresponds to the shear modulus of elasticity, and \(b = EI\) signifies the cross-sectional rigidity, where \(E\) is Young's modulus, \(G\) is the modulus of rigidity, \(\kappa\) is the transverse shear factor, and \(I\) is the moment of inertia. The functions \(S = k (\varphi_x + \psi)\) and \(M = b \psi_x\) represent the shear force and bending moment, respectively. Additionally, \(\delta(x - \xi)\) denotes the Dirac delta function with a unit mass at \(x = \xi\). 
{The nonlinear functions  $F$ and $G$  represent either the intrinsic nonlinear stiffness of the system or an external static control force.}
Finally, \(\gamma_1\) and \(\gamma_2\) are positive damping coefficients.
Subscripts $x$ and $t$ represent partial derivatives with respect to $x$ and $t$.
%


%
\begin{center}
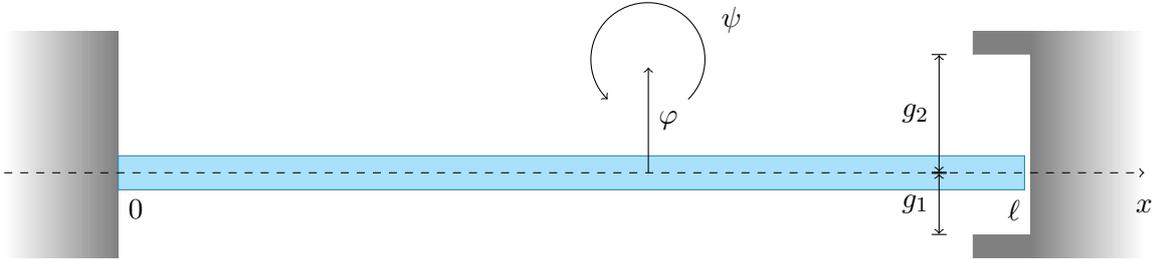
\begin{figure}[!ht]
\begin{tikzpicture}[scale=1.5]
\shade[shading=axis,shading angle=90] (6,-0.6) rectangle +(1,2);
\shade [shading=axis,shading angle=-90] (-3,-0.6) rectangle +(1,2);
\filldraw[fill=cyan!30!white, draw=cyan!60!black] (-2,0) rectangle +(7.95,0.3);
\filldraw[fill=black!50!white, draw=black!50!white] (5.5,-0.4) rectangle +(0.5,-0.2);
\filldraw[fill=black!50!white, draw=black!50!white] (5.5,1.2) rectangle +(0.5,0.2);
\draw[dashed,->] (-3,0.15) -- (7,0.15);
\draw[|<->|] (5.2,0.15) -- node[left] {$g_2$} (5.2,1.2);
\draw[|<->|] (5.2,0.15) -- node[left] {$g_1$} (5.2,-0.4);
\fill (-2,0)  node[below right] {$0$};
\fill (6,0)  node[below left] {$\ell$};
\fill (7,0)  node[below] {$x$};
\draw[->] (2.65,0.15) -- node[right] {$\varphi$} (2.65,1.08);
\draw[->] (3,0.8) arc [start angle=-45, end angle=225, radius=0.5] ;
\fill (3.2, 1.5) node[right] {$\psi$};
\end{tikzpicture}
\\
\caption{Beam subjected to a constraint at the free end $x=\ell$.}
\end{figure}\label{FIG.1}
\end{center}
Before proceeding, let us recall pertinent results from the literature on contact problems \cite{RiveraArantes,Barber,BRN,Eck,Kikuchi,KS,RiveraHigidio,Wriggers}.
This list is by no means exhaustive and it is intended solely to offer a concise overview of the developments achieved to date in this field.

In this paper, we propose a novel and nonstandard approach to the analysis of the Signorini problem by coupling the linear Timoshenko system with a dynamic boundary condition described by an ordinary differential equation. This leads to a hybrid PDE–ODE formulation, where the interaction is mediated by a coupling parameter $\epsilon$; see system \eqref{Tr1.1} below.
  We employ semigroup theory to establish the well-posedness of the problem and to prove the exponential stability of the associated system. We arrive at the contact problem with a normal compliance condition by introducing a Lipschitz perturbation of the original system. Finally, setting $\epsilon\rightarrow 0$ we get the Signorini problem.
This approach is made possible by the observability inequalities inherent to the Timoshenko model. We argue that this method is more effective than the standard penalty technique (see, e.g., \cite{RiveraArantes,KS,RiveraHigidio} and the references therein)
as it provides more accurate insights into the asymptotic behavior of the solution. In particular, we show that the boundary conditions have no influence on the long-time dynamics. This implies that the exponential decay result holds for arbitrary boundary conditions—unlike in \cite{RiveraArantes,BRN,RiveraHigidio}, where specific boundary conditions were essential to establish exponential stability.

The remainder of this manuscript is organized as follows. In Section~\ref{sec-semilin}  we prove that an abstract semilinear problem is well-posed and we show that, under suitable conditions, there exists a global compact attractor.
Section~\ref{sec2} is devoted to the study of a linear hybrid model  approaching the penalized problem, associated to \eqref{SignoriniProblem}--\eqref{Hip.Stress}. In particular, we find an useful observability result for the Timoshenko model.
 In Section~\ref{sec4} we consider some applications to contact problems with normal compliance condition.
Finally in Section~\ref{sec-Signorini} we analyze a Signorini-type problem and we establish the existence of a global compact attractor.

\section{Semigroup and General Results}\label{sec-semilin}
\setcounter{equation}{0}

In this section we establish the well-posedness of an abstract semilinear problem and demonstrate, under appropriate conditions, the existence of a global compact attractor. 

To proceed, the following hypotheses must be formulated.

Let $\mathcal{H}$ be a Hilbert space with $(\cdot,\cdot)_{\mathcal{H}}$ and $\|\cdot\|_{\mathcal{H}}$ its inner product and norm, respectively.

Let us denote by  $ \mathcal {F} $  a  local  Lipschitz function on $ \mathcal {H} $ verifying 
\begin{equation}\label{ff0}
\mathcal{F}(0)=0\in \mathcal {H} .
\end{equation}
We suppose that
for any ball $ B_R=\{W\in\mathcal{H}:\;\; \|W\|_{\mathcal{H}}\leq R\} $, there exists a function globally of Lipschitz $\widetilde{\mathcal{F}_R}$ such that
\begin{equation}\label{ff1}
\mathcal{F}(U)=\widetilde{\mathcal{F}_R}(U),\quad \forall U\in B_R
\end{equation}
and additionally, that there exists a positive constant $\kappa_0$ such that
\begin{equation}\label{ff2}
\int_0^t\big(\widetilde{\mathcal{F}_R}(U(s)),U(s)\big)_{\mathcal{H}}\;ds \leq \kappa_0 \|U(0)\|_{\mathcal{H}}^2,\qquad \forall U\in C([0,T];\mathcal{H})
\end{equation}

Under these conditions we present
\begin{Th}\label{Lipschitz}
Let \( \{T(t)\}_{t \geq 0} \) be a \( C_0 \)-semigroup of contractions, exponentially stable, with infinitesimal generator \( \mathbb{A} \) on a Hilbert space \( \mathcal{H} \). Let \( \mathcal{F} : \mathcal{H} \to \mathcal{H} \) be a locally Lipschitz continuous function satisfying conditions \eqref{ff0}, \eqref{ff1} and \eqref{ff2}
. Then, for any initial condition \( U_0 \in \mathcal{H} \), there exists a unique global mild solution to the abstract Cauchy problem
\begin{equation}
\label{absCE}
U_t - \mathbb{A} U = \mathcal{F}(U), \quad U(0) = U_0,
\end{equation}
which decays exponentially to zero as \( t \to \infty \). Moreover if \( U_0 \in D(\mathcal{A}) \) then the mild solution is the strong solution of \eqref{absCE}. 
\end{Th}

\begin{proof} By hypotheses, there exist positive constants $c_0$ and $\gamma$ such that
$
\|T(t)\|\leq c_0e^{-\gamma t},
$
and  $\widetilde{\mathcal{F}_R}$ globally Lipschitz with Lipschitz constant $K_0$ verifying conditions (\ref{ff1}) and   (\ref{ff2}). Let us suppose that $U_0\in D(\mathbb{A})$, 
{the final result will then follow by standard density arguments, since the dissipative nature of $\mathbb{A}$ allows us to extend the conclusions to arbitrary initial data in the phase space $\mathcal{H}$.}

It is well know that  there exists only one global mild solution to
\begin{equation}\label{absCE2}
U_{t}^R-\mathbb{A}U^R=\widetilde{\mathcal{F}_R}(U^R),\quad U^R(0)=U_0\in D(\mathbb{A}).
\end{equation}
Since the phase space \( \mathcal{H} \) is reflexive, then we have that $U^R$ is a strong solution to \eqref{absCE2}, see \cite[p.~189, Theorem 1.6]{pazy}. Hence  $U^R\in L^\infty(0,T;D(\mathbb{A}))$.
Multiplying  equation \eqref{absCE2} by $U^R$ we get that
$$
\frac 12 \frac{d}{dt}\|U^R(t)\|_{\mathcal{H}}^2-(\mathbb{A}U^R, U^R)_{\mathcal{H}}=(\widetilde{\mathcal{F}_R}(U^R),U^R)_{\mathcal{H}}.
$$
Since the semigroup is contractive, its infinitesimal generator is dissipative, therefore
$$
\|U^R(t)\|_{\mathcal{H}}^2\leq \|U_0\|_{\mathcal{H}}^2+2\int_0^t(\widetilde{\mathcal{F}_R}(U^R),U^R)_{\mathcal{H}}\;dt.
$$
Using (\ref{ff2}) we get
\begin{equation}\label{BoundIn}
\|U^R(t)\|_{\mathcal{H}}^2\leq (1+k_0)\|U_0\|_{\mathcal{H}}^2.
\end{equation}
Note that for $R> (1+k_0)\|U_0\|_{\mathcal{H}}^2$, hypothesis \eqref{ff1} yields 
$$
\widetilde{\mathcal{F}_R}(V)=\mathcal{F}(V),\quad \text{for any $V$ with}\;\; \|V\|_{\mathcal{H}}\leq R.
$$
In particular $U^R\in B_R$ so we have 
$$
\widetilde{\mathcal{F}_R}(U^R(t))=\mathcal{F}(U^R(t)).
$$
This means that $U^R$ is also solution of system (\ref{absCE}) and because of the uniqueness we conclude that $U^R=U$. Hence $U$ is the global solution of (\ref{absCE}).

To show the exponential stability to system (\ref{absCE}), it is enough to show the exponential decay to system (\ref{absCE2}). To do that,  we use fixed points arguments.
Let us consider 
$$
\mathcal{T}(V)=T(t)U_0+\int_0^{t}T(t-s)\widetilde{\mathcal{F}_R}(V(s))\;ds.
$$
and let us define the set 
$$
E_{\mu}=\left\{V\in L^\infty(0,\infty;\mathcal{H});\;\; t\mapsto e^{\mu t}\|V(s)\|\in L^\infty(\R)\right\}.
$$

Note that $\mathcal{T}$ is invariant over $E_{\gamma-\varepsilon}$ for $\varepsilon$ small and $\gamma-\varepsilon>0$. In fact, for any $V\in E_{\gamma-\varepsilon}$ we have
\begin{eqnarray*}
\|\mathcal{T}(V)\|_{\mathcal{H}}&\leq&\|U_0\|_{\mathcal{H}}e^{-\gamma t}+\int_0^t\|\widetilde{\mathcal{F}_R}(V(s))\|_{\mathcal{H}}e^{-\gamma(t-s)}\;ds,\\
&\leq&\|U_0\|_{\mathcal{H}}e^{-\gamma t}+K_0\int_0^t\|V(s)\|_{\mathcal{H}}e^{-\gamma(t-s)}\;ds,\\
&\leq&\|U_0\|_{\mathcal{H}}e^{-\gamma t}+K_0e^{-\gamma t}\int_0^te^{\varepsilon s}\;ds \sup_{s\in [0,t]}\left\{ e^{(\gamma-\varepsilon) s}\|V(s)\|_{\mathcal{H}}\right\},\\
&\leq&\|U_0\|_{\mathcal{H}}e^{-\gamma t}+\frac{K_0C}{\varepsilon}e^{-(\gamma -\varepsilon)t}.
  \end{eqnarray*}
Hence $\mathcal{T}(V)\in E_{\gamma-\varepsilon}$.
Using standard arguments we can show that $\mathcal{T}^n$ satisfies
\begin{eqnarray}\label{idd2}
\|\mathcal{T}^n(W_1)-\mathcal{T}^n(W_2)\|\leq \frac{(k_1t)^n}{n!}\|W_1-W_2\|_{\mathcal{H}}.
  \end{eqnarray}
Where {$k_1$ is the Lipschitz constant of $\mathcal{F}_R$}. Therefore we have a unique fixed point satisfying

$$
\mathcal{T}^n(U)=U=T(t)U_0+\int_0^{t}T(t-s)\widetilde{\mathcal{F}_R}(U(s))\;ds,
$$
that is $U$ is a solution of  (\ref{absCE2}), and since $\mathcal{T}$ is invariant over $E_{\gamma-\varepsilon}$, then the solution decays exponentially.

\end{proof}

\begin{Cor}\label{LipschitzhG} Let $\mathbb{A}$ as in Theorem \ref{Lipschitz} and 
let $\mathcal{F}$ satisfying conditions \eqref{ff0}--\eqref{ff2}.
Let us denote by $\mathcal{F}_1(U)=\mathcal{F}(U)+F_0$ with $F_0\in \mathcal{H}$. 
 Then there exists a global solution to 
 \begin{equation}
\label{absCEF1}
U_t - \mathbb{A} U = \mathcal{F}_1(U), \quad U(0) = U_0,
\end{equation}
verifying that  for any $R>0$ there exists $t_0=t_0(R)$ such that the solution of \eqref{absCEF1} satisfies 
\begin{equation}\label{absor}
\|U(t)\|_{\mathcal{H}}\leq  c_1e^{-\mu t}R+c_2\|F_0\|_{\mathcal{H}},\quad \forall t\geq t_0
\end{equation}
for any $\|U_0\|_{\mathcal{H}}\leq R$.
\end{Cor}
\begin{proof} 
 Let us denote by 
$$
\widetilde{E}_{\mu}=\left\{V\in L^\infty(0,\infty;\mathcal{H});\;\; \|V(s)\|_{\mathcal{H}}\leq   2c_0e^{-\mu t}\|U_0\|_{\mathcal{H}}+c\|F_0\|_{\mathcal{H}}\right\}.
$$
To do that,  we use fixed points arguments.
Let us consider the operator $\mathcal{T}$ given by 
$$
\mathcal{T}(V)=T(t)U_0+\int_0^{t}T(t-s)\widetilde{\mathcal{F}_R}(V(s))\;ds.
$$
First note that 
$$
\mathcal{T}(0)=T(t)U_0+\int_0^{t}T(t-s)F_0\;ds.
$$
Since $T(t)=e^{\mathbb{A}t}$ is exponentially stable, we have that 
$$
\|\mathcal{T}(0)\|_{\mathcal{H}}\leq c_0e^{-\gamma t}\|U_0\|_{\mathcal{H}}+t\|F_0\|_{\mathcal{H}}.
$$
Analogously, 
\begin{eqnarray*}
\mathcal{T}^2(0)&=&T(t)U_0+\int_0^{t}T(t-s)\left[T(s)U_0+\int_0^{s}T(s-\sigma )F_0\;d\sigma\right]\;ds\\
&=&T(t)U_0+\int_0^{t}T(t)U_0\;ds+\int_0^{t}T(t-s)\int_0^{s}T(s-\sigma )F_0\;d\sigma\;ds\\
&=&T(t)U_0+tT(t)U_0+\int_0^{t}T(t-s)\int_0^{s}T(s-\sigma )F_0\;d\sigma\;ds\\
\end{eqnarray*}
Since $T$ is a contraction semigroup we have that 
\begin{eqnarray*}
\|\mathcal{T}^2(0)\|_{\mathcal{H}}\leq c_0(1+t)e^{-\gamma t}\|U_0\|_{\mathcal{H}}+\frac{t^2}{2}\|F_0\|.
\end{eqnarray*}
Repeating the above procedure we have 
\begin{eqnarray*}
\|\mathcal{T}^3(0)\|_{\mathcal{H}}\leq c_0\left(1+t+\frac{t^2}{2}\right)e^{-\gamma t}\|U_0\|_{\mathcal{H}}+\frac{t^3}{3!}\|F_0\|.
\end{eqnarray*}
In general we have 
\begin{eqnarray*}
\|\mathcal{T}^n(0)\|_{\mathcal{H}}\leq c_0\left(1+t+\frac{t^2}{2}+\cdots +\frac{t^n}{n!}\right)e^{-\gamma t}\|U_0\|_{\mathcal{H}}+\frac{t^n}{n!}\|F_0\|.
\end{eqnarray*}
Using \eqref{idd2} for $W_1=V$ and $W_2=0$ we arrive to

\begin{eqnarray*}
\|\mathcal{T}^n(V)\|&\leq &\|\mathcal{T}^n(0)\|+\frac{(k_1t)^n}{n!}\|V\|_{\mathcal{H}}\\
&\leq &c_0\left(1+t+\frac{t^2}{2}+\cdots +\frac{t^n}{n!}\right)e^{-\gamma t}\|U_0\|_{\mathcal{H}}+\frac{t^n}{n!}\|F_0\|+\frac{(k_1t)^n}{n!}\|V\|_{\mathcal{H}}.\\
\end{eqnarray*}
Therefore for $n$ large we have that $\mathcal{T}^n(V)\in E_{\mu}$. 
\end{proof}

\bigskip
\noindent
The semigroup $T(t)$ admits the decomposition 
$$
T(t)=L(t)+N(t),
$$
such that, for a given initial data $U_0\in B_R(0)$ we set  
$$
T(t)U_0=U(t),\quad L(t)U_0=W(t),\quad N(t)U_0=V(t), 
$$
where 
\begin{equation}\label{absCEw1}
W_{t}-\mathbb{A}W=\mathcal{F}(W),\quad W(0)=U_0\in\mathcal{H},
\end{equation}
Note that from Theorem \ref{Lipschitz} there exists a solution $W$ to system \eqref{absCEw1} which decays exponentially. 
Hence the function $V=U-W$, where $U$ is the solution of \eqref{absCE}, satisfies 
\begin{equation}\label{absCEv}
V_{t}-\mathbb{A}V=\mathcal{F}_1(U)-\mathcal{F}(W),\quad V(0)=0\in\mathcal{H}.
\end{equation}
So we have that $V$ is globally defined.

\noindent 
Let us denote by $\mathfrak{B}$ a bounded operator on $\mathcal{H}$ such that $\mathfrak{B}U_t\in \mathcal{H}$. 
That is 
\begin{equation}\label{frakB}
\|\mathfrak{B}U_t\|_{\mathcal{H}}\leq c\|U\|_{\mathcal{H}}.
\end{equation}
We have that 
$$
\mathfrak{B}:C([0,T];\mathcal{H})\rightarrow C([0,T];\mathcal{H}).
$$
\begin{Rem}
As an example of an operator \(\mathfrak{B}\), consider the semigroup generated by the system described in \eqref{TTr1.1}. The corresponding state variable of the abstract model is defined as \( U = (\varphi, \varphi_t, \psi, \psi_t) \). We introduce the operator \(\mathfrak{B}\) acting on \( U \) as follows:

\[
\mathfrak{B}U = (0, \varphi, 0, \psi)^{\top}.
\]
It is straightforward to verify that this operator satisfies the condition specified in \eqref{frakB}.

\end{Rem}

\bigskip
\noindent
Additionally to \eqref{ff1}--\eqref{ff2} we assume that $\mathcal{F}$ is continuos differentiable and satisfies, 
\begin{equation}\label{Freg1}
\mathcal{F}(W)=\widetilde{\mathcal{F}}(\mathfrak{B}W),\quad \|D\mathcal{F}(Y)\|\leq c\|Y\|_{\mathcal{H}},\quad \forall \,Y\in B_R\subset \mathcal{H},
\end{equation}
\begin{equation}\label{Freg2}
\frac{\partial }{\partial t}\mathcal{F}(U),\quad \frac{\partial }{\partial t}\mathcal{F}(W)\in C(0,T;\mathcal{H}),
\end{equation}
where $D$ denotes the derivative of $\mathcal{F}$ and $R>0$. Using  \cite[p.~109, Corollary 4.2.11]{pazy}, we have that the solution of system \eqref{absCEF1} verifies 
$$
V\in C^1(0,T;\mathcal{H})\cap C(0,T;D(\mathcal{A})).
$$
Since the initial condition of the system \eqref{absCEv} vanishes, we have:
$$
\lim_{t\rightarrow 0} V(t)=0,\quad \lim_{t\rightarrow 0} \mathcal{F}_1(U)= \mathcal{F}(U_0)+F_0,\quad \lim_{t\rightarrow 0} \mathcal{F}(W)= \mathcal{F}(U_0) .
$$
Using the above limit in \eqref{absCEv} we get that 
$$
\lim_{t\rightarrow 0} V_t(t)=F_0.
$$
Differentiating equation \eqref{absCEv} we get 
$$
V_{tt}-\mathbb{A}V_t=D\mathcal{F}(U)\mathfrak{B} U_t-D\mathcal{F}(W) \mathfrak{B}W_t,\quad V_t(0)=F_0\in\mathcal{H},
$$
so we have 
\begin{equation}\label{absCEvDf}
V_{tt}-\mathbb{A}V_t=D\mathcal{F}(U)\mathfrak{B}V_t-\left[D\mathcal{F}(W)-D\mathcal{F}(U)\right]\mathfrak{B} W_t\quad V_t(0)=F_0\in\mathcal{H},
\end{equation}

\begin{Le}\label{lemaW}
Under the above conditions of $\mathcal{F}$ and assuming additionally that $\mathcal{F}$ is differentiable and $\mathbb{A}$ an infinitesimal generator of exponentially stable semigroup, then  the solution of \eqref{absCEv} satisfies $V\in C(0,T;D(\mathbb{A}))$ and
 $$
 \|V_t(t)\|_{\mathcal{H}}+
 \|V(t)\|_{D(\mathcal{A})}\leq c (\|F\|_{\mathcal{H}}+\|U_0\|_{\mathcal{H}}),\quad \forall t\geq t_0.
 $$
 
\end{Le}
\begin{proof}
Let us denote by 
$\mathfrak{W}=V_t$ the solution of 
\begin{equation}\label{SolD}
\mathfrak{W}_{t}-\mathbb{A}\mathfrak{W}=D\mathcal{F}(U)\mathfrak{B}V_t-\left[D\mathcal{F}(W)-D\mathcal{F}(U)\right]\mathfrak{B}W_t\quad \mathfrak{W}(0)=F_0\in\mathcal{H}.
\end{equation}
So we have 
\begin{equation}\label{WWW}
\mathfrak{W}(t)=e^{\mathbb{A}t}F_0+\int_0^te^{\mathbb{A}(t-s)}\left\{D\mathcal{F}(U)\mathfrak{B}V_t-\left[D\mathcal{F}(W)-D\mathcal{F}(U)\right]\mathfrak{B}W_t\right\}ds.
\end{equation}
Using \eqref{frakB} and since $U$ and $W$ are bounded for any $t\geq t_0$ we have:
\begin{eqnarray}
\|D\mathcal{F}(U)\mathfrak{B}V_t\|_{\mathcal{H}}&\leq &c\|U\|_{\mathcal{H}}\|V\|_{\mathcal{H}},\\
\noalign{\medskip}
\|(D\mathcal{F}(W)-D\mathcal{F}(U))\mathfrak{B}W_t\|_{\mathcal{H}}
&\leq &c\|U\|_{\mathcal{H}}\|W\|_{\mathcal{H}}.
\end{eqnarray}
Using Corollary \eqref{LipschitzhG} and  Theorem \eqref{Lipschitz},      we conclude that 
$$
\|V\|_{\mathcal{H}}\leq c\left(\|U_0\|_{\mathcal{H}}+\|F\|_{\mathcal{H}}\right),\quad \|W\|_{\mathcal{H}}\leq c\|U_0\|_{\mathcal{H}}e^{-\gamma t}.
$$
Inserting the above inequalities into \eqref{WWW} and recalling that $\mathbb{A}$ is exponentially stable, we get
\begin{equation}\label{WWW1}
\|\mathfrak{W}(t)\|_{\mathcal{H}}\leq c\|F_0\|_{\mathcal{H}}e^{-\gamma t}+c\int_0^te^{-\gamma (t-s)}\left(\|U\|_{\mathcal{H}}\|V\|_{\mathcal{H}}+\|U\|_{\mathcal{H}}\|W\|_{\mathcal{H}}\right)ds.
\end{equation}
Since $U$, $V$ and $W$ are bounded, we find 
\begin{equation*}
\|\mathfrak{W}(t)\|_{\mathcal{H}}\leq c\|F_0\|_{\mathcal{H}}e^{-\gamma t}+c\int_0^te^{-\gamma (t-s)}\left(\|U_0\|_{\mathcal{H}}+\|F\|_{\mathcal{H}}\right)ds.
\end{equation*}
From where we have 
\begin{equation*} 
\|\mathfrak{W}(t)\|_{\mathcal{H}}\leq c(\|U_0\|+\|F\|).
\end{equation*}
Using equation \eqref{absCEv} we conclude that 
$$
\|\mathbb{A}V(t)\|_{\mathcal{H}}^2\leq \|F_0\|_{\mathcal{H}}^2+c\|U_0\|_{\mathcal{H}}^2,\quad \forall t\geq 0.
$$
 Consequently, our conclusion follows.
\end{proof}

\begin{Th} \label{Attractor} Let us denote by $S(t)$ the semigroup defined by the abstract equation 
\eqref{absCE}. Suppose additionally that the immersion $D(\mathbb{A})\subset \mathcal{H}$ is compact,  then $S(t)$ possesses a unique compact  global attractor $\mathfrak{A}$ contained in $D(\mathbb{A})$. 
\end{Th}
\begin{proof}
Corollary \ref{LipschitzhG} provides the existence of a bounded absorving set $\mathbb{B}$, while Lemma 
 \ref{lemaW} show that $S(t)\mathbb{B}$ is exponentially attracted by a bounded set $\mathcal{C}\subset D(\mathbb{A})$. Hence $\mathcal{C}$ is a compact attracting set. By standard arguments of the theory of dynamical  systems (see \cite{Gatti,MR2271373,Hale,Teman}), we conclude that there exists a compact global attractor $\mathfrak{A}\subset \mathcal{C}$. 
\end{proof}

We conclude this section with the following useful characterization.
\begin{Th}\label{TNovo}  Let $S(t)=e^{{\mathbb{A}}t}$ be a
			$C_0$-semigroup of contractions on Banach space. Then, $S(t)$ is
			exponentially stable if and only if
			\begin{equation}\label{hyp}i\mathbb{R}\subset \varrho(\mathbb{A})
			\quad\text{and}\quad
			\omega_{ess}(S(t))<0,
			\end{equation}
where $\omega_{ess}(S(t))$ is the essential growth bound of the semigroup $S(t)$.
		\end{Th}
		\begin{proof}Here we use \cite[Corollary~2.11, p.~258]{EngelNagel} establishing that the type $\omega$ of the semigroup $e^{\mathbb{A}t}$ verifies
			\begin{equation}\label{idxx}
			\omega=\max\{\omega_{ess}, \omega_\sigma(\mathbb{A})\},
			\end{equation}
			where 
$\omega_\sigma(\mathbb{A})$ is the upper bound of the spectrum of $\mathbb{A}$. 	Moreover, for any $c>\omega_{ess}$, the set $\mathcal{I}_c:=\sigma(\mathbb{A})\cap\{\lambda\in \mathbb{C}:\;\; \mbox{Re}\lambda \geq c\}$ is  finite.

Let us suppose that \eqref{hyp} is valid.
Since the essential type of the semigroup $\omega_{ess}$ is negative, identity \eqref{idxx} states that the type of the semigroup will be negative provided $\omega_\sigma(\mathbb{A})<0$. 

If
$ \omega_\sigma (\mathbb {A}) \leq \omega_ {ess} $ then we have nothing to prove. Let us suppose that $ \omega_ \sigma (\mathbb {A})> \omega_ {ess} $.  From \eqref{hyp} and  Hille-Yosida Theorem  we have $ \overline{\mathbb {C} _ +} \subset \varrho (\mathbb {A}) $, hence
$ \omega_\sigma (\mathbb {A})\leq 0 $.  On the other hand 
$ \mathcal{I}_ {\omega_{ess}+\delta} $ is finite for $\delta>0$ verifying  $ \omega_ {ess} + \delta <0 $ and $ \omega_ {ess} + \delta <\omega_ \sigma (\mathbb {A}) $. Therefore  we have 
$$
\omega_ \sigma (\mathbb {A})=\sup \mbox{Re }\sigma(\mathbb{A})=\sup \mbox{Re } \mathcal{I}_ {\omega_{ess}+\delta}<0.
$$ 
Hence, the sufficient condition follows.

Reciprocally, let us suppose that the semigroup $S(t)$ is exponentially stable, in particular it goes to zero. Then, by   \cite[Theorem~1.1]{DuyBatt}  we have that $i\mathbb{R}\subset \varrho(\mathbb{A})$. Moreover, since the type $\omega$ verifies \eqref{idxx}, we have that
			$$
			\omega_{ess}\leq \max\{\omega_{ess}, \omega_\sigma(\mathbb{A})\}=\omega<0.
			$$
			Then, our conclusion follows.
\end{proof}
Note that the above characterization is valid for any Banach space.

\section{The  Hybrid Model}\label{sec2}
\setcounter{equation}{0}

In order to apply the semigroup theory to study the Signorini problem, we consider first the linear hybrid model, approaching the penalized problem, associated to \eqref{SignoriniProblem}--\eqref{Hip.Stress}. For details to pass from the Signorini problem to the penalized one, see, e.g., \cite{BRN}.  
\begin{eqnarray}\label{Tr1.1}
\begin{array}{ll}
 \rho_1\,\varphi_{tt} - \kappa \left(\varphi_x+\psi\right)_x+\gamma_1\delta(x-\xi)\varphi_t =0 & \text{in $I\times (0,+\infty)$,}
 \\
 \noalign{\medskip}
 \rho_2\,\psi_{tt} -b\,\psi_{xx} + \kappa \left(\varphi_x+\psi\right)+\gamma_2\delta(x-\xi)\psi_t =0 & \text{in $I\times (0,+\infty)$,}\\
  \noalign{\medskip}
  \epsilon \varphi_{tt}(\ell,t) +\epsilon \varphi_{t}(\ell,t) +\epsilon \varphi(\ell,t)+S(\ell,t)=0& \text{in $ (0,+\infty)$,}
\end{array}
\end{eqnarray}
satisfying the boundary conditions
\begin{eqnarray}\label{XXeqNs}
 \begin{array}{llll}
\varphi(0,t)=0,&  \psi_x(0,t)=0, &\psi(\ell,t)=0 & \text{in $(0,+\infty)$.}
 \end{array}
\end{eqnarray}
Note that $\varphi(\ell,t):=v(t)$ is determined by equation \eqref{Tr1.1}$_3$.
This dynamic boundary condition can be interpreted as a beam rigidly attached at the end $x=\ell$ to a tip body of mass $\epsilon$ that models a sealed container with a granular material, for example
sand. This granular material dampens the movement of the system by internal friction (for details see, e.g., \cite{Tip3,Tip2,Tip1}).

To formulate system \eqref{Tr1.1} within the semigroup framework, we transform the above system as a transmission problem. Indeed, let us denote by
 $I$ the open set
$$
I=(0,\xi)\cup (\xi,\ell) .
$$
\newcommand{\salto}[2]{[\![#1]\!]_{\xi_{#2}}}
Let us introduce the notation $\salto{f}{}$ means  the  jump of $f$ in $\xi$. That is 
$$
\salto{f}{}:=f(\xi^+)-f(\xi^-).
$$
Therefore,  it is easy to see that system \eqref{Tr1.1} is equivalent to
\begin{eqnarray}\label{NewTr1.1}
\begin{array}{ll}
 \rho_1\,\varphi_{tt} - \kappa \left(\varphi_x+\psi\right)_x =0 & \text{in $I\times (0,+\infty)$,}
 \\
 \noalign{\medskip}
 \rho_2\,\psi_{tt} -b\,\psi_{xx} + \kappa \left(\varphi_x+\psi\right) =0 & \text{in $I\times (0,+\infty)$,}\\
  \noalign{\medskip}
  \epsilon \varphi_{tt}(\ell,t) +\epsilon \varphi_{t}(\ell,t) +\epsilon \varphi(\ell,t)+S(\ell,t)=0& \text{in $ (0,+\infty)$,}
\end{array}
\end{eqnarray}
with 
\begin{eqnarray}\label{eqNs}
 \begin{array}{llll}
\varphi(0,t)=0,&  \psi_x(0,t)=0, &\psi(\ell,t)=0 & \text{in $(0,+\infty)$.}
 \end{array}
\end{eqnarray}

\begin{equation}\label{trxx}
\salto{\varphi}{}=\salto{\psi}{}=0,\quad \salto{k\varphi_x}{}=\gamma_1\varphi_t(\xi,t),\quad \salto{b\psi_x}{}=\gamma_2\psi_t(\xi,t)
\end{equation}

and the initial conditions
\begin{equation}\label{Tr1.2} 
\begin{array}{lc}
& \varphi(x,0)=\varphi_0(x), \quad
 \varphi_t(x,0)=\varphi_1(x), \quad
 \psi(x,0)=\psi_0(x), \quad
 \psi_t(x,0)=\psi_1(x),\\
 \noalign{\medskip}
 & v(0)=v_0,\quad v_t(0)=v_1,\quad
 \end{array}
\end{equation}
This physically admissible coupling  \eqref{trxx} represents the continuity of displacement and the discontinuity of force at $x=\xi$. We can observe that
if $\gamma_i=0$, $i=1,2$, then there is not energy dissipation at $x=\xi$ and the linkage at $x=\xi$ is conservative. Instead, if $\gamma_i>0$, $i=1,2$, then the linkage is dissipative, as the case under consideration.
Putting $\Phi=\varphi_t$, $\Psi=\psi_t$ and $V=v_t$, the phase space of our problem is
$$
\mathcal{H}=V_0 \times L^{2}(0,\ell)\times V_\ell \times L^{2}(0,\ell)\times \mathbb{C}^2,
$$
where
$$
V_0=\left\{w\in H^{1}(0,\ell):\;\; w(0)=0\right\} \quad \text{and}\quad V_\ell=\left\{w\in H^{1}(0,\ell):\;\; w(\ell)=0\right\},
$$
with the norm
\[
 \|(\varphi,\Phi,\psi,\Psi,v,V) \|^2_{\mathcal{H}}=\int_0^\ell
\left( \kappa|\varphi_x+\psi|^2+\rho_1|\Phi|^2
 +b|\psi_x|^2+\rho_2|\Psi|^2\right)\;dx+\epsilon|v|^2+\epsilon|V|^2.
\]
\subsection{The $C_0$ semigroup of contractions}
Denoted by $B^\top$ the transpose of a matrix $B$ and introducing the state vector
\[
 U(t)=\left(\varphi(t),\Phi(t),\psi(t),\Psi(t),v(t),V(t)\right)^\top:=\left(\mathcal{U},\mathcal{V}\right)^\top,
\]
where $\mathcal{U}=\left(\varphi(t),\Phi(t),\psi(t),\Psi(t)\right)^\top$, $\mathcal{V}=\left(v(t),V(t)\right)^\top$  the  transmission conditions are given by
\begin{equation}\label{wxx}
\salto{k\varphi_x}{}=\gamma_1\Phi(\xi)\qquad \text{and} \qquad \salto{b\psi_x}{}=\gamma_2\Psi(\xi),
\end{equation}
Hence, system \eqref{Tr1.1}--\eqref{Tr1.2} can be written as a linear ODE in $\mathcal{H}$ of the form
\begin{equation}\label{ODE}
 \frac{d}{dt}U(t)=\mathcal{A}\,U(t),
\end{equation}
where the domain $\mathcal{D}(\mathcal{A})$ of the linear operator $\mathcal{A}:D(\mathcal{A})\subset \mathcal{H} \to \mathcal{H}$ is given by
\begin{equation}\label{Domain}
 \mathcal{D}(\mathcal{A})=\left\{
 U\in \mathcal{H}: \;\varphi,\psi \in H^2(I),\;
 (\Phi,  \Psi)  \in V_0\times V_\ell,\;
\text{verifying  \eqref{wxx}}
 \right\},
\end{equation}
and
\begin{equation}\label{IGA}
 \mathcal{A}U=\begin{bmatrix}
 \Phi \\
 \noalign{\medskip}
 \displaystyle\frac{\kappa}{\rho_1}\left(\varphi_x+\psi\right)_x
 \\
 \noalign{\medskip}
 \Psi
 \\
 \noalign{\medskip}
 \displaystyle\frac{b}{\rho_2}\,\psi_{xx}-\frac{\kappa}{\rho_2}\left(\varphi_x+\psi\right)
  \\
 \noalign{\medskip}
V
 \\
 \noalign{\medskip}
-V-v-\frac 1\epsilon S(\ell,t)
 \end{bmatrix}.
\end{equation}
Straightforward  calculations shows that the operator 
 $\mathcal{A}$ is dissipative. Indeed, for every $U \in \mathcal{D}(\mathcal{A})$,
\begin{equation}\label{diss}
 \operatorname{Re}\langle \mathcal{A} U(t),U(t)\rangle_{\mathcal{H}}=
  -\gamma_1 |\Phi(\xi,t)|^2  -\gamma_2 |\Psi(\xi,t)|^2 -\epsilon|V(t)|^2 \leq 0.
\end{equation}
Considering the resolvent equation
\begin{equation}\label{Resolv}
i\lambda U-\mathcal{A}U=F,
\end{equation}
and taking inner product with $U$ over the phase space $\mathcal{H}$, we get
\begin{equation}\label{diss2}
 \gamma_1 |\Phi(\xi,t)|^2 +\gamma_2 |\Psi(\xi,t)|^2 +\epsilon|V(t)|^2=\operatorname{Re}\langle  U(t),F(t)\rangle_{\mathcal{H}}.
\end{equation}

Using standard procedures we can show that $0\in\varrho(\mathcal{A})$. According to 
Lummer-Phillips Theorem \cite[Theorem~1.2.4 ]{s4LC99}  the operator $\mathcal{A}$ is the infinitesimal generator of a contraction semigroup $
\mathcal{T}(t):=e^{t\mathcal{A}}: \mathcal{H} \to \mathcal{H}
$. See also \cite[Theorem~1.4.3]{pazy}.
So we have 
\begin{Th}
For any ${U}_0\in \mathcal{H}$ there exists a unique mild solution
\begin{equation}\label{semi}
 U(t)=\left(\varphi(t),\varphi_t(t),\psi(t),\psi_t(t),v(t),V(t)\right)^\top=\mathcal{T}(t)\,U_0.
\end{equation}
to problem \eqref{Tr1.1}. Moreover if the initial data
$U_0\in D(\mathcal{A})$ there exists a strong solution satisfying
\[U \in C^1(0,T;\mathcal{H})\cap C(0,T;D(\mathcal{A})).
\] 
\end{Th}
\hfill $\square$

\subsection{Comparison between the hybrid and
the non hybrid model}
Let us  make a comparison between the hybrid and
the non hybrid Timoshenko model given by

\begin{eqnarray}\label{TTr1.1}
\begin{array}{ll}
 \rho_1\,\varphi_{tt} - \kappa \left(\varphi_x+\psi\right)_x=0 & \text{in $I\times (0,+\infty)$,}
 \\
 \noalign{\medskip}
 \rho_2\,\psi_{tt} -b\,\psi_{xx} + \kappa \left(\varphi_x+\psi\right)=0 & \text{in $I\times (0,+\infty)$,}\\
  \noalign{\medskip}
\end{array}
\end{eqnarray}
satisfying the boundary conditions
\begin{eqnarray}\label{T1eqNs}
 {\varphi}(0,t)= {S}(\ell,t)=0,\quad   {\psi}_x(0,t)= {\psi}(\ell,t)=0,\quad \text{in $(0,+\infty)$},
\end{eqnarray}
where $S=\kappa(\varphi_x+\psi)$ and $M=b\psi_x$. 
Here Theorem \ref{TNovo}  will play an important role.

Let us denote the infinitesimal generator of  system \eqref{TTr1.1}--\eqref{T1eqNs} by $\mathcal{A}_{T}$ where
\begin{equation}\label{IGA-T}
 \mathcal{A}_{T}\,\mathcal{U}=\begin{bmatrix}
 \Phi \\
 \noalign{\medskip}
 \displaystyle\frac{1}{\rho_1}S_{x}
 \\
 \noalign{\medskip}
 \Psi
 \\
 \noalign{\medskip}
 \displaystyle\frac{1}{\rho_2}\,M_{x}-\frac{1}{\rho_2}S
 \end{bmatrix}.
\end{equation}
The phase space we consider for the above model is
\begin{equation}\label{Phase1}
\mathbf{H}=V_0 \times L^{2}(0,\ell)\times V_\ell \times L^{2}(0,\ell),
\end{equation}
Hence the domain $\mathcal{D}(\mathcal{A}_{T})$ of the linear operator $\mathcal{A}_{T}:D(\mathcal{A}_{T})\subset \mathbf{H} \to \mathbf{H}$ is given by
\[
 \mathcal{D}(\mathcal{A}_{T})=\left\{
 \mathcal{U}\in \mathbf{H}: \;\varphi,\psi \in H^2(I),\;
 (\Phi,  \Psi)  \in V_0\times V_\ell,\;
\text{verifying \eqref{trxx}  and  \eqref{T1eqNs}}
 \right\}.
\]
Similarly as the hybrid model, we have
\begin{equation}\label{visco2}
 \operatorname{Re}\langle \mathcal{A}_{T} U,U\rangle_{\mathbf{H}}=
   -\gamma_1 |\Phi(\xi,t)|^2  -\gamma_2 |\Psi(\xi,t)|^2 \leq 0,
\end{equation}
with $\gamma_1,\gamma_2> 0$.
Under this notations we get that system \eqref{TTr1.1}--\eqref{T1eqNs} can be written as
\begin{equation}\label{ODE-T}
 \frac{d}{dt}\mathcal{U}(t)=\mathcal{A}_{T}\,\mathcal{U}(t).
\end{equation}

Let us denote by $\mathbf{T}(t)=e^{t\mathcal{A}_{T}}$ the semigroup associated to system \eqref{TTr1.1}--\eqref{T1eqNs}.

The primary objective of this section is to demonstrate that the semigroup \(\mathbf{T}\) is exponentially stable if and only if the semigroup \(\mathcal{T}\), associated with the hybrid model, is also exponentially stable. This result indicates that the dissipation introduced by the ordinary differential equation (ODE) in \eqref{Tr1.1}$_3$, part of the hybrid model, does not significantly affect the stability. Consequently, the exponential stability of the system persists as \(\epsilon \to 0\).

Let us introduce the space
\begin{equation}\label{Phase2}
\widetilde{\mathbf{H}}=\mathbf{H}\times \{0\}\times \{0\},
\end{equation}
intended as the extended phase space.
Let us denote by $\Pi$ the projection of $\mathcal{H}$ onto $\widetilde{\mathbf{H}}$:
\begin{equation}\label{Phase3}
\Pi(\varphi,\Phi,\psi,\Psi,v,V) =(\varphi,\Phi,\psi,\Psi,0,0).
\end{equation}
Let us decompose the infinitesimal generator $\mathcal{A}$ in the following way
\begin{equation}\label{DefAT}
\mathcal{A}
:=
\begin{pmatrix}
\mathcal{A}_{T}&\mathbf{0}_{4\times 2}\\
B&K
\end{pmatrix},
\end{equation}
with
$$
B=\begin{pmatrix}
0&0&0&0\\
\frac{\kappa}{\epsilon}\boldsymbol{\gamma}_1&0&0&0
\end{pmatrix},\quad
K=\begin{pmatrix}
0&I\\
-I&-I
\end{pmatrix},
$$
where $\boldsymbol{\gamma}_1\varphi=\varphi_x(\ell)$.  Hence, recalling that $U:=(\mathcal{U}, \mathcal{V})^\top$, where $\mathcal{U}:=(\varphi,\Phi,\psi,\Psi)$ and $\mathcal{V}:=(v,V)$, we get
$$
\mathcal{A}U=\begin{pmatrix}
\mathcal{A}_{T}\mathcal{U}\\
B\mathcal{U}+K\mathcal{V}
\end{pmatrix}=\begin{pmatrix}
\mathcal{A}_{T}\mathcal{U}\\
0
\end{pmatrix}+\begin{pmatrix}
0\\
K\mathcal{V}
\end{pmatrix}+\begin{pmatrix}
\mathbf{0}\\
B\mathcal{U}
\end{pmatrix},\quad \forall U\in D(\mathcal{A}_i).
$$
Under the above conditions we can state the following Lemma:
 \begin{Le}\label{lemma-T00}
The difference $\mathcal{T}(t)-\mathbf{T}(t)\Pi$ is a compact operator over $\mathcal{H}$. Hence the corresponding essential types $\omega_{\text{ess}}(\mathcal{T})$ and $\omega_{\text{ess}}(\mathbf{T}(t)\Pi)$  are equal.
\end{Le}
\begin{proof}
Note that the solution of $U_t-\mathcal{A}U=0$, $U(0)=U_0$ can be written as
$$
\begin{pmatrix}
\mathcal{U}\\
\mathcal{V}
\end{pmatrix}_t=\begin{pmatrix}
\mathcal{A}_{T}\mathcal{U}\\
0
\end{pmatrix}+\begin{pmatrix}
0\\
K\mathcal{V}
\end{pmatrix}+\begin{pmatrix}
\mathbf{0}\\
B\mathcal{U}
\end{pmatrix},
$$
with $ U_0=(\mathcal{U}_0, \mathcal{V}_0)^\top$ which implies that
$$
\mathcal{U} =e^{t\mathcal{A}_{T}}\mathcal{U}_0,\quad \text{and}\quad\mathcal{V} =e^{tK}\mathcal{V}_0+\int_0^te^{(t-s)K}B\mathcal{U}(s)\;ds.
$$
Therefore
\begin{equation}\label{eeqq1} 
U(t)-\begin{pmatrix}
e^{t\mathcal{A}_{T}}\mathcal{U}_0\\
0
\end{pmatrix}=\begin{pmatrix}
\mathbf{0}\\
e^{tK}\mathcal{V}_0+\int_0^te^{(t-s)K}B\mathcal{U}(s)\;ds.
\end{pmatrix}.
\end{equation}
Since 
$$
\mathfrak{G}(t)=\int_0^te^{(t-s)K}B\mathcal{U}(s)\;ds\in H^1(0,T),
$$
we conclude that the right hand side of \eqref{eeqq1} is a compact operator, 
therefore
$$
\left[\mathcal{T}(t)-\mathbf{T}(t)\Pi\right]
$$
is a compact operator. Hence, we arrive at the conclusion.
\end{proof}

 \begin{Th}\label{T00}
The semigroup $\mathcal{T}(t)$ is exponentially stable if and only if  $ \mathbf{T}(t)$ is exponentially stable. \end{Th}
\begin{proof}
Because of Lemma \ref{lemma-T00} to show the exponential stability of   $\mathcal{T}(t)$ or  $ \mathbf{T}(t)$ thanks to Theorem \ref{TNovo} it is enough to show that the imaginary axes is contained over the resolvent set of the corresponding infinitesimal generator. Let us suppose that  $\mathcal{T}(t)$ is exponentially stable and suppose that  $ \mathbf{T}(t)$ is not exponentially stable. So this implies that $i\mathbb{R}\not\subset \varrho(\mathcal{A}_T)$, resolvent of $\mathcal{A}_T$. Therefore,  there exists an eigenvector $0\not = W\in D(\mathcal{A}_T)$ and $0\not = \lambda \in\mathbb{R}$ such that 
$$ 
\mathcal{A}_TW=i\lambda W.
$$
Taking $\widetilde{W}=(W,0,0)$ we conclude that $\widetilde{W}$ is also an eigenvector of $\mathcal{A}$, that is  $\mathcal{A}\widetilde{W}=i\lambda \widetilde{W}$ but this is contradictory to our hypothesis.  

Let us suppose now that $ \mathbf{T}(t)$ is exponentially stable and suppose that $ \mathcal{T}(t)$ is not. This implies that there exists an eigenvector $\widetilde{W}\not = 0$ and $0\not = \lambda \in\mathbb{R}$ such that  such that 
$$
\mathcal{A}\widetilde{W}=i\lambda\widetilde{W}.
$$
Multiplying by $\widetilde{W}$ we get 
$$
i\lambda(\widetilde{W},\widetilde{W})_{\mathcal{H}}- (\mathcal{A}\widetilde{W},\widetilde{W} )_{\mathcal{H}}=0.
$$
Taking the real part and using \eqref{diss}, we conclude that $\widetilde{W}=(\widetilde{\varphi},\widetilde{\Phi},\widetilde{\psi},\widetilde{\Psi},v,V)$ satisfies 
\begin{equation}\label{dissx3}
 \gamma_1 |\Phi(\xi,t)|^2 +\gamma_2 |\Psi(\xi,t)|^2 +\epsilon|V(t)|^2=0.
\end{equation}
This, in particular, implies that $v=V=0$. Hence, the vector ${W}=(\widetilde{\varphi},\widetilde{\Phi},\widetilde{\psi},\widetilde{\Psi})$ is an eigenvector of $\mathcal{A}_T$ which is a contradiction. So, our conclusion follows. 
\end{proof}

\noindent 
Now we are in conditions to establish the exponential stability. 
\begin{Th}\label{esstipe} Let us suppose that $\xi\in \mathbb{Q}\ell$ such that $\xi\ne\frac{2n}{2m+1}\ell$,   $\forall n,m\in \mathbb{N}$, with 
$2n$, $2m+1$ co-prime numbers then  
 the semigroup $\mathcal{T}(t)=e^{\mathcal{A}t}$ associated to system \eqref{NewTr1.1}-\eqref{Tr1.2} 
is exponentially stable .
\end{Th}
\begin{proof}
Using the result of \cite {TimoPoint} and Theorem \ref{T00} we conclude that the hybrid model \eqref{NewTr1.1}-\eqref{Tr1.2} is exponentially stable. 
\end{proof}

We will finish this section showing the observability result to Timoshenko model \eqref{NewTr1.1}. To do this,
let us introduce the functionals

\begin{eqnarray}
\mathcal{I}(x,t)&=&\rho_2b|\psi_t(x,t)|^2 +|M(x,t)|^2+\rho_1\kappa|\varphi_t(x,t)|^2 +|S(x,t)|^2,\label{II1}
\\
\noalign{\medskip}
\mathcal{L}(t)&=&\int_0^\ell \left( \rho_2bq_x|\psi_t|^2 +q_x|M|^2+\rho_1\kappa q_x|\varphi_t|^2 +q_x|S|^2  \right)dx \nonumber\\
\noalign{\medskip}
&&- \int_0^L \left(q \rho_1\kappa\varphi_t \overline{\psi_t}- qS\overline{M}\right)dx,\label{II2}
\end{eqnarray}
where $q$ is as
 in (\ref{def-q}). Hence,
there exist positive constants $C_0$ and $C_1$ such that
\begin{equation}\label{equivxx}
C_0\int_0^\ell \mathcal{I}(x,t)\;dx\leq \mathcal{L}(t)\leq C_1\int_0^\ell \mathcal{I}(x,t)\;dx.
\end{equation}
Under the above conditions we have
\begin{Le}\label{observability}
The solution of system \eqref{NewTr1.1} satisfies
$$
\left|\int_0^tq(\ell)\mathcal{I}(\ell,t)\;dx-\int_0^t\mathcal{L}(s)\;ds\right|\leq cE(0),
$$
$$
\left|\int_0^tq(0)\mathcal{I}(0,t)\;dx-\int_0^t\mathcal{L}(s)\;ds\right|\leq cE(0).
$$
\end{Le}
\begin{proof}
Let us multiply equation (\ref{NewTr1.1})$_2$  by $q\overline{M}$
we get
\begin{equation}\label{Obpsi}
\frac{d}{dt}\int_0^\ell\rho_2q\psi_tM\;dx-\frac 12\int_0^\ell q\frac{d}{dx}\left[\rho_2b|\psi_t|^2+|M|^2 \right]dx=-\int_0^\ell q\overline{M}S\;dx.
\end{equation}
Similarly, multiplying equation (\ref{NewTr1.1})$_1$ by $q\overline{S},$ we get
\begin{eqnarray}\label{Obpsi1}
\frac{d}{dt}\int_0^\ell\rho_1q\varphi_tS\;dx-\frac 12\int_0^\ell q\frac{d}{dx}\left[\rho_2\kappa|\varphi_t|^2+|S|^2 \right]dx&=&
\rho_1\kappa\int_0^Lq\varphi_t\psi_t\;dx.
\end{eqnarray}
Therefore summing identities \eqref{Obpsi} and \eqref{Obpsi1} and integrating over $[0,t]$ we get
$$
\frac 12\int_0^t\int_0^\ell q\frac{d}{dx}\mathcal{I}(x,t)\;dxdt=\left.\int_0^\ell\left(\rho_1q\varphi_tS+\rho_2q\psi_tM\right)
\;dx\right|_0^t-
\int_0^t\int_0^L\rho_1\kappa q\varphi_t\psi_t-q\overline{M}S\;dx,
$$
performing integrations by parts and recalling the definition of $\mathcal{L}$, we get
$$
\int_0^t\left[q(\ell)\mathcal{I}(\ell,s)-q(0)\mathcal{I}(0,s)\right]\;ds
-\int_0^t\mathcal{L}(s)\;ds
=\left.\int_0^\ell\left(\rho_1q\varphi_tS+\rho_2q\psi_tM\right)dx\right|_0^t.
$$
Since
$$
\left|\int_0^\ell\rho_2q\psi_tM\;dx\right|\leq cE(0),\quad \left|\int_0^\ell\rho_1q\varphi_tS\;dx\right|\leq cE(0),
$$
we conclude that
$$
\left|\int_0^t\left[q(\ell)\mathcal{I}(\ell,s)-q(0)\mathcal{I}(0,s)\right]ds-\int_0^t\mathcal{L}(s)\;ds\right|\leq cE(0),
$$
taking
\begin{equation}\label{def-q}
q(x)= \frac{e^{nx}-1}{n},\qquad q_0(x)= \frac{e^{-nx}-e^{-n\ell}}{n}.
\end{equation}

It should be noted that $q'(x)$ is substantially larger than $q(x)$ for sufficiently large values of $n$.
Therefore, there exist positive constants $c_0$ and $c_1$ such that
$$
c_0\int_a^b\mathcal{I}(x)\;dx\leq \mathcal{L} \leq c_1 \int_0^L\mathcal{I}(x)\;dx.
$$
So
our result follows.
\end{proof}

\section{Applications to Contact  Problem with normal compliance}\label{sec4}
\setcounter{equation}{0}

Here we will consider the contact problem with normal compliance, given by 
   \begin{eqnarray}\label{NorComp}
\begin{array}{ll}
 \dis\rho_1\varphi_{tt}-k(\varphi_x+\psi)_x=-F(\varphi)
,& \text{in}\ I\times(0,\infty),\\
\noalign{\medskip}
 \dis\rho_2\psi_{tt}-b\psi_{xx}+k(\varphi_x+\psi)=-G(\psi)
,& \text{in}\ I\times(0,\infty),\\
\noalign{\medskip}
S(L,t)=-d_2
\left[(u(L,t)-g_2)^{+}\right] ^p+d_1\left[(g_1-u(L,t))^{+}\right] ^p & \text{in}\ (0,\infty),
\end{array}
 \end{eqnarray}
 verifying the boundary condition \eqref{eqNs} and the initial conditions.  
The normal compliance condition is given by condition \eqref{NorComp}$_3$ where \( S(L,t) \) is the normal traction, \( u(L,t) \) is the normal displacement, \( g_i \) is the initial gap, \( d_i > 0 \) is the stiffness coefficient. The power \( p \) governs the nonlinearity of the contact response. The most prevalent powers are \( p = 1 \), \( p = 2 \) and  \( p = 3 \). 
The linear compliance \( p = 1 \) is a linear reaction force, resembling a Hookean spring. This conditions is 
 used in  Mechanical Engineering: gears, bearings,  elastic beams on rigid supports; or in Civil Engineering: standard foundations (see \cite{Kikuchi,Eck}). Quadratic Compliance \( p = 2 \) is applied to model quadratic reaction force, modeling nonlinear hardening, with applications to vehicle Dynamics: tire-road contact (see \cite{Barber}).
The cubic compliance \( p = 3 \) or cubic reaction force is suitable for large deformations, with applications to 
Biomechanics: articular tissues, prostheses (see \cite{Glowinski,Wriggers}).
 
 We first establish the exponential stability and the existence of global attractor to the hybrid model

   \begin{eqnarray}\label{PenSystemiip}
\begin{array}{ll}
 \dis\rho_1\varphi^{\epsilon}_{tt}-k(\varphi^{\epsilon}_x+\psi^{\epsilon})_x=-F(\varphi^{\epsilon})
,& \text{in}\ I\times(0,\infty),\\
\noalign{\medskip}
 \dis\rho_2\psi^{\epsilon}_{tt}-b\psi^{\epsilon}_{xx}+k(\varphi^{\epsilon}_x+\psi^{\epsilon})=-G(\psi^{\epsilon})
,& \text{in}\ I\times(0,\infty),\\
\noalign{\medskip}
 \epsilon v_{tt}^{\epsilon} +\epsilon v_t^{\epsilon} +\epsilon v^{\epsilon}+S^\epsilon(L,t)=-d_2
\left[(v^{\epsilon}-g_2)^{+}\right] ^p+d_1\left[(g_1-v^{\epsilon})^{+}\right] ^p & \text{in}\ (0,\infty),
\end{array}
 \end{eqnarray}
 Denoting by  $\mathcal{F}$ the function
 \begin{eqnarray}\label{fff}
\mathcal{F}(U)=(0,F(\varphi),0,G(\psi),0,f(v))^\top, 
\end{eqnarray}
where 
$$
f(v)=d_2
\left[(v^{\epsilon}-g_2)^{+}\right] ^p-d_1\left[(g_1-v^{\epsilon})^{+}\right] ^p.
$$
It is not dificult to see that conditions  (\ref{ff0})--(\ref{ff2}) are verified. In fact, $\mathcal{F}(0)=0$. Moreover

\begin{eqnarray*}
\int_0^t\big({\mathcal{F}}(U(s)),U(s)\big)_{\mathcal{H}}\,ds &=&-\int_0^tf(v)v_t \,ds
-\int_0^t\int_0^{\ell}\left[F(\varphi)\varphi_t +G(\psi)\psi_t\right] \,dx\,ds\\
 &=&-\int_0^t\frac{d}{dt}\widehat{f}(v) \;ds -\int_0^t\int_0^{\ell}\frac{d}{dt}[\widehat{F}(\varphi)+\widehat{G}(\psi)]dx\,ds\\
&\leq& \widehat{f}(v_0)+\int_0^{\ell}\widehat{F}(\varphi_0)\,dx+\int_0^{\ell}\widehat{G}(\psi_0)\,dx,
\end{eqnarray*}
where
$$
\widehat{f}=\frac{d_2}{p+1}\left[(v^{\epsilon}-g_2)^{+}\right] ^{p+1}+\frac{d_1}{p+1}\left[(g_1-v^{\epsilon})^{+}\right] ^{p+1},
$$
$$
\widehat{F}(\varphi)=\int_0^{\varphi}F(s)\;ds,\qquad \widehat{G}(\psi)=\int_0^{\psi}F(s)\;ds.
$$
Moreover, if we consider 
$$
F(s)=\mu_1s|s|^{\alpha},\quad G(s)=\mu_2s|s|^{\beta},
$$
 using the mean value Theorem, we obtain the inequality
$$
\Big|s|s|^{\alpha}-r|r|^{\alpha}\Big |\leq (|s|^\alpha+|r|^\alpha)|s-r|.
$$
For $R>0$ there exists the cut-off function for $F$ given by 
\begin{eqnarray*}
F_{1,R}=\left\{
\begin{array}{ll}
\mu_1 x|x|^{\alpha}& \text{if $x\leq R$},\\
\noalign{\medskip}
\mu_1 x|R|^{\alpha}&\text{if $x\geq R$},
\end{array}
\right.
\end{eqnarray*}
Similarly for $G$ and $f$. Note that the function $\mathcal{F}$ defined in \eqref{fff} verifies conditions \eqref{ff0}--\eqref{ff2} and \eqref{Freg1}--\eqref{Freg2}.

\noindent
Next we show the energy inequality
\begin{Le}\label{Energy}
The solution of system \eqref{NewTr1.1} satisfies
\begin{eqnarray}\label{estimate}
 E(t,\varphi^{\epsilon},\psi^{\epsilon})+  \int_0^t\left[ \gamma_1|\varphi_t^{\epsilon}(\xi,t)|^2  +\gamma_2 |\psi_t^{\epsilon}(\xi,t)|^2\right]dt\leq
  E(0,\varphi^{\epsilon},\psi^{\epsilon}),
\end{eqnarray}
where
$$
2E(t)=
 \int_0^\ell\bigg[\rho_1|\varphi_t^{\epsilon}|^2+\rho_2|
\psi_t^{\epsilon}|^2+k|\varphi_x^{\epsilon} + \psi^{\epsilon}|^2 +b|\psi_x^{\epsilon}|^2
\bigg]dx+\mathcal{N}_p(t)+ \epsilon|v_t^{\epsilon}|^2+ \epsilon|v^{\epsilon}|^2,  \label{semigroup 2.6}
$$
and
$$
\mathcal{N}_p(t):=\frac{d_2}{p+1}|(\varphi^{\epsilon}(\ell,t)-g_2)^{+}|^{p+1}+\frac{d_2}{p+1}|(g_1-\varphi^{\epsilon}(\ell,t))^{+}|^{p+1}.
$$
\end{Le}
\begin{proof}
 Multiplying equation (\ref{NewTr1.1})$_1$ by $\varphi_t,$ equation (\ref{NewTr1.1})$_2$ by $\psi_t,$ and equation (\ref{NewTr1.1})$_3$ by $v_t,$ summing up the product result our conclusion follows.
\end{proof}

Under the above conditions we have that

\begin{Th}\label{ExpN}
For any initial condition \( U_0 \in \mathcal{H} \), there exists a unique global mild solution to the  semilinear semigroup defined by hybrid system  \eqref{PenSystemiip} that  is exponentially stable.  Moreover if \( U_0 \in D(\mathcal{A}) \) then the mild solution is the strong solution of \eqref{absCE}. 
\end{Th}

\begin{proof}
   It is a direct consequence of  Theorem \ref{Lipschitz}. 
\end{proof}
Moreover it is not difficult to see that the function $\mathcal{F}$ defined in \eqref{fff} verifies conditions 
\eqref{ff0}--\eqref{ff2} and  \eqref{Freg1}--\eqref{Freg2}. So we have the following result.

\begin{Th} \label{Attractor2} Let us denote by $T_\epsilon(t)$ the semigroup defined by the semilinear hybrid system \eqref{PenSystemiip},  then $T_\epsilon(t)$ possesses a unique compact global attractor $\mathfrak{A}_\epsilon$ contained in $D(\mathcal{A})$. 
\end{Th}
\begin{proof}
By Theorem \ref{Attractor} we only need to check that the immersion $D(\mathcal{A})\subset \mathcal{H}$ is compact. But this is an immediate consequence of \eqref{Domain} and the phase space $\mathcal{H}$. 
\end{proof}

Now we are in conditions to establish our main result to the normal compliance contact problem. 

\begin{Th} \label{Attractor2bis}
 Let us denote by $T(t)$ the semigroup defined by problem  \eqref{NorComp}, then
for any initial data
$ (\varphi_0,\varphi_1,\psi_0,\psi_1)\in \mathcal{H}$ there exists a global mild solution to
 problem  \eqref{NorComp}, the semigroup $T(t)$ is exponentially stable provided  $\mathcal{F}(0)=0$. Moreover it  possesses a unique compact global attractor $\mathfrak{A}_0$ contained in $D(\mathcal{A})$. 
\end{Th}
\begin{proof}
From Theorem \ref {Lipschitz} we have that there exists only one solution to system (\ref {PenSystemii}). Using Lemma \ref{Energy} and Lemma \ref{observability} we get
\begin{eqnarray}\label{III}
\mathcal{I}_\epsilon(\ell,t) \quad\text{uniformly bounded in  } \quad L^2(0,T),
\end{eqnarray}
where  $\mathcal{I}_\epsilon(x,t) =\mathcal{I}(x,t,\varphi^\epsilon,\psi^\epsilon)$ given by  \eqref{II1}. 
This means that the first order energy is uniformly bounded for any $\epsilon>0$. Standard procedures implies that the solution of system \eqref{PenSystemii} converges in the distributional sense to system
(\ref{NorComp}). It remains only to show that conditions (\ref{Cond.Signorini}) holds.  To do that
we use  the observability inequality in Lemma \ref{observability}, and we get that $\varphi_t^\epsilon(\ell,t)$ and $S^\epsilon(\ell,t)$  are  bounded in $L^2(0,T)$, so is $v_{tt}^\epsilon$. Using (\ref {PenSystemiip})$_4$ we obtain
$$
\int_0^T\left[\epsilon v_{tt} +\epsilon v_t +\epsilon v+S^\epsilon(\ell,t)\right]\omega(t)\, dt=-\int_0^T\left[-d_2
\left[(v^{\epsilon}-g_2)^{+}\right] ^p+d_1\left[(g_1-v^{\epsilon})^{+}\right] ^p\right]\omega(t)\, dt,
$$
for any $\omega\in C_0^\infty(0,T)$.
It is no difficult to see that
$$
\lim_{\epsilon\rightarrow0}\int_0^T\left(\epsilon v_{tt}^\epsilon +\epsilon v_t^\epsilon +\epsilon v^\epsilon\right)\omega(t)\, dt=0.
$$
In fact, from (\ref{PenSystemii})$_4$  $\epsilon v_{tt}^\epsilon$ is bounded for any $\epsilon>0$  (by a constant depending on $\epsilon$) in $L^2(0,T)$, from (\ref{III}) $v_{t}^\epsilon$ is also uniformly bounded in $L^2(0,T)$. Therefore $v_{t}^\epsilon$ is a continuous function, uniformly bounded in $L^\infty(0,T)$. Making an integration by parts we find
$$
\int_0^T\epsilon v_{tt}^\epsilon\omega(t)\, dt= \epsilon \left.v_{t}^\epsilon\omega\right|_0^T-\int_0^T\epsilon v_{t}^\epsilon\omega_t(t)\, dt\quad\rightarrow\quad 0.
$$
Hence,
\begin{eqnarray*}
\lim_{\epsilon\rightarrow0}\int_0^TS^\epsilon(\ell,t)\omega(t)\, dt&=&
\lim_{\epsilon\rightarrow0}\int_0^T\left[-d_2
(v^{\epsilon}-g_2)^{+}\right] ^p+d_1\left[(g_1-v^{\epsilon})^{+}\right]\omega(t)\, dt.
\end{eqnarray*}
Using the strong convergence of $v^{\epsilon}$ we conclude that 
$$
\int_0^T\left[S(\ell,t)\right]\omega(t)\, dt=-\int_0^T\left[-d_2
\left[(v-g_2)^{+}\right] ^p+d_1\left[(g_1-v)^{+}\right] ^p\right]\omega(t)\, dt,
$$
for any  $\omega\in C_0^\infty(0,T)$. So, we have that 
$$
S(\ell,t) =-d_2
\left[(v-g_2)^{+}\right] ^p+d_1\left[(g_1-v)^{+}\right] ^p.
$$

From this relation we obtain (\ref{Hip.Stress}). The proof of the existence is now complete. To show the asymptotic behaviour, we consider
$$
 E(t,\varphi^{\epsilon},\psi^{\epsilon})\leq  E(0,\varphi^{\epsilon},\psi^{\epsilon})e^{-\gamma t}.
 $$
Integrating over $[t_1,t_2]$ and applying the semicontinuity of the norm, we conclude the exponential stability of a solution of the Signorini problem. 

Finally, we show the existence of a global compact attractor. Let us consider the projection operator $\Pi$
\eqref{Phase2}, where  
$$
\Pi:\mathcal{H}\rightarrow\widetilde{\mathbf{H}}.
$$
So we have that $\Pi(\mathfrak{A}_\epsilon)$ is a compact set of  $\widetilde{\mathbf{H}}$. Let us denote by $\epsilon_n\rightarrow0$ and set $\mathfrak{A}_0=\cap_{n=1}^\infty\Pi(\mathfrak{A}_{\epsilon_n})$. It is not difficult to see that $\mathfrak{A}_0$  is compact.
$$
\lim_{t\rightarrow\infty}\mbox{dist}(\Pi(T_\epsilon(t)U_0), \mathfrak{A}_\epsilon)=0,\quad \forall \epsilon>0
$$
Hence  $\mathfrak{A}_0$ is a compact absorbing set for the semigroup $T_0(t)=T(t)$. 
By standard arguments of the theory of dynamical systems, we conclude that there exists a compact global attractor $\widetilde{\mathfrak{A}}_0$. 
\end{proof}

\section{Signorini problem}\label{sec-Signorini}
Here we will consider the Signorini contact problem given by system \eqref{SignoriniProblem}, with $F$ and $G$ as in the above section. 

The proof of this theorem is based on the hybrid approximation given by the system 

\begin{eqnarray}\label{PenSystemii}
\begin{array}{ll}
 \dis\rho_1\varphi^{\epsilon}_{tt}-k(\varphi^{\epsilon}_x+\psi^{\epsilon})_x=-F(\varphi^{\epsilon})
,& \text{in}\ I\times(0,\infty),\\
\noalign{\medskip}
 \dis\rho_2\psi^{\epsilon}_{tt}-b\psi^{\epsilon}_{xx}+k(\varphi^{\epsilon}_x+\psi^{\epsilon})=-G(\psi^{\epsilon})
,& \text{in}\ I\times(0,\infty),\\
\noalign{\medskip}
 \epsilon v_{tt}^{\epsilon} +\epsilon v_t^{\epsilon} +\epsilon v^{\epsilon}+S^\epsilon(L,t)=-\frac{1}{\epsilon}
(v^{\epsilon}-g_2)^{+}+\frac{1}{\epsilon}(g_1-v^{\epsilon})^{+} & \text{in}\ (0,\infty),
\end{array}
 \end{eqnarray}
 By applying the same arguments as in the preceding section, we can establish the following result for the hybrid model. 
\begin{Th} \label{Attractor3} Let us denote by $T_\epsilon(t)$ the semilinear semigroup defined by problem  \eqref{PenSystemiip},  then $T_\epsilon(t)$ possesses a unique compact global attractor $\mathfrak{A}_\epsilon$ contained in $D(\mathcal{A})$. 
\end{Th}

We are now in a position to establish our main result concerning the Signorini problem.

\begin{Th} \label{Attractor3bis} 
 For any initial data
$ (\varphi_0,\varphi_1,\psi_0,\psi_1)\in \mathcal{H}$ there exists a mild solution to Signorini problem \eqref{SignoriniProblem}--\eqref{Cond.Signorini} which decays as established in Theorem \ref{ExpN}, provided $\mathcal{F}(0)=0$.
 Moreover it  possesses a unique compact global attractor $\mathfrak{A}_0$ contained in $D(\mathcal{A})$. 
\end{Th}

\begin{proof}
Using the same arguments as in Theorem \ref{Attractor2} we  only have show that conditions (\ref{Cond.Signorini}) holds.  To do that
we use  the observability inequality in Theorem \ref{observability}, and we get that $\varphi_t^\epsilon(\ell,t)$ and $S^\epsilon(\ell,t)$  are  bounded in $L^2(0,T)$, so is $v_{tt}^\epsilon$. Using (\ref {PenSystemii})$_3$ we obtain
$$
\int_0^T\left[\epsilon v_{tt} +\epsilon v_t +\epsilon v+S^\epsilon(\ell,t)\right](u-v)\, dt=-\frac{1}{\epsilon}\int_0^T\left[(v-g_2)^{+}-(g_1-v)^{+}\right](u-v)\, dt.
$$
For any $u\in L^2(0,T;\mathcal{K})\cap H^1(0,T;L^2(0,\ell))$, where $\mathcal{K}=\{w\in H^1(0,\ell),\;\; g_1\leq u(x)\leq g_2\}.$
It is no difficult to see that
$$
\lim_{\epsilon\rightarrow0}\int_0^T\left(\epsilon v_{tt}^\epsilon +\epsilon v_t^\epsilon +\epsilon v^\epsilon\right)(u-v^\epsilon)\, dt=0.
$$
In fact, from (\ref{PenSystemii})$_4$  $\epsilon v_{tt}^\epsilon$ is bounded for any $\epsilon>0$  (by a constant depending on $\epsilon$) in $L^2(0,T)$, from (\ref{III}) $v_{t}^\epsilon$ is also uniformly bounded in $L^2(0,T)$. Therefore $v_{t}^\epsilon$ is a continuous function, uniformly bounded in $L^\infty(0,T)$. Making an integration by parts we find
$$
\int_0^T\epsilon v_{tt}^\epsilon[u-v^\epsilon]\, dt= \epsilon \left.v_{t}^\epsilon[u-v^\epsilon]\right|_0^T-\int_0^T\epsilon v_{t}^\epsilon[u_t-v_t^\epsilon]\, dt\quad\rightarrow\quad 0.
$$
Hence,
\begin{eqnarray*}
\lim_{\epsilon\rightarrow0}\int_0^TS^\epsilon(\ell,t)[u(t)-v(t)]\, dt&=&
\lim_{\epsilon\rightarrow0}\int_0^T-\frac{1}{\epsilon}\left[(v-g_2)^{+}
-(g_1-v)^{+}\right][u(t)-v(t)]\, dt.
\end{eqnarray*}
Since
\begin{eqnarray*}
\int_0^T(v-g_2)^{+}[u(t)-v(t)]\, dt
&=&\int_0^T(v-g_2)^{+}[u(t)-g_2]\;dt-\int_0^T(v-g_2)^{+}(v-g_2)\, dt\\
&=&\int_0^T(v-g_2)^{+}[u(t)-g_2]\;dt-\int_0^T(v-g_2)^{+}(v-g_2)^+\, dt\leq 0,
\end{eqnarray*}
for all $g_1\leq u\leq g_2$.
Similarly we get
\begin{eqnarray*}
-\int_0^T[(g_1-v)^{+}[u(t)-v(t)]\, dt&\leq&0.
\end{eqnarray*}
Therefore, from the last two inequalities we get
$$
\int_0^T\frac{1}{\epsilon}\Big[(v-g_2)^{+}-(g_1-v)^{+}\Big][u(t)-v(t)]\, dt\leq 0, \quad \forall \epsilon>0.
$$
For any $u\in H^1(0,T;L^2(0,\ell))$ such that $g_1\leq u\leq g_2$. Taking the limit $\epsilon\rightarrow 0$ we get
$$
\int_0^TS(\ell,t)[u(\ell,t)-\varphi(\ell,t)]\, dt
\geq 0,\quad \forall u\in L^2(0,T;\mathcal{K}).
$$
From this relation we obtain (\ref{Hip.Stress}). The proof of the existence is now complete. To show the asymptotic behaviour, let us consider 
$$
 E(t,\varphi^{\epsilon},\psi^{\epsilon})\leq  E(0,\varphi^{\epsilon},\psi^{\epsilon})e^{-\gamma t}.
 $$
Integrating over $[t_1,t_2]$ and applying the semicontinuity of the norm, we conclude the exponential stability of a solution of the Signorini problem.
Finally, we show the existence of a global compact attractor. Let us consider the projection operator $\Pi$
\eqref{Phase2}, where  
$$
\Pi:\mathcal{H}\rightarrow\widetilde{\mathbf{H}}.
$$
So we have that $\Pi(\mathfrak{A}_\epsilon)$ is a compact set of  $\widetilde{\mathbf{H}}$. 
Let us denote by $\epsilon_n\rightarrow0$ and set $\mathfrak{A}_0=\cap_{n=1}^\infty\Pi(\mathfrak{A}_{\epsilon_n})$. It is not difficult to see that $\mathfrak{A}_0$  is a compact. Since 

$$
\lim_{t\rightarrow\infty}\mbox{dist}(\Pi(T_\epsilon(t)U_0), \mathfrak{A}_\epsilon)=0,\quad \forall \epsilon>0
$$

This implies that $\mathfrak{A}_0$ is a compact absorbing set for the semigroup $T_0(t)=T(t)$. 
By standard arguments of the theory of dynamical systems we conclude that there exists a compact global attractor $\widetilde{\mathfrak{A}}_0$.
\end{proof}

\begin{Rem}
The uniqueness of the solution to Signorini problem \eqref{SignoriniProblem}--\eqref{Cond.Signorini} remains an open
question.
\end{Rem}

\centerline{\textsc{Funding}}

{\small J.E. Mu\~{n}oz Rivera would like to thank CNPq project 307947/2022-0 and  Fondecyt Proyect 1230914, for the financial support.
M.G. Naso has been partially supported by Gruppo Nazionale per la Fisica Matematica (GNFM) of the Istituto Nazionale di Alta Matematica
(INdAM).}

\medskip

\centerline{\textsc{Conflict of interest}}

{\small This work does not have any conflict of interest.}

\bibliographystyle{amsplain}
\bibliography{bibl.bib}

\providecommand{\bysame}{\leavevmode\hbox to3em{\hrulefill}\thinspace}
\providecommand{\MR}{\relax\ifhmode\unskip\space\fi MR }
\providecommand{\MRhref}[2]{%
  \href{http://www.ams.org/mathscinet-getitem?mr=#1}{#2}
}
\providecommand{\href}[2]{#2}
\begin{thebibliography}{10}

\bibitem{Tip3}
K.~T. Andrews, J.~R. Fern\'andez, and M.~Shillor, \emph{A thermoviscoelastic
  beam with a tip body}, Comput. Mech. \textbf{33} (2004), no.~3, 225--234.

\bibitem{Tip2}
K.~T. Andrews and M.~Shillor, \emph{Vibrations of a beam with a damping tip
  body}, Math. Comput. Modelling \textbf{35} (2002), no.~9-10, 1033--1042.

\bibitem{RiveraArantes}
S.~Arantes and J.~E. Mu\~noz Rivera, \emph{Exponential decay for a
  thermoelastic beam between two stops}, J. Thermal Stresses \textbf{31}
  (2008), no.~6, 537--556.

\bibitem{BV}
A.~V. Babin and M.~I. Vishik, \emph{Attractors of evolution equations}, Studies
  in Mathematics and its Applications, vol.~25, North-Holland Publishing Co.,
  Amsterdam, 1992, Translated and revised from the 1989 Russian original by
  Babin.

\bibitem{Barber}
J.~R. Barber, \emph{Contact mechanics}, Solid Mechanics and its Applications,
  vol. 250, Springer, Cham, 2018.

\bibitem{DuyBatt}
C.~J.~K. Batty and T.~Duyckaerts, \emph{Non-uniform stability for bounded
  semi-groups on {B}anach spaces}, J. Evol. Equ. \textbf{8} (2008), no.~4,
  765--780.

\bibitem{BRN}
A.~Berti, J.~E. Mu\~noz Rivera, and M.~G. Naso, \emph{A contact problem for a
  thermoelastic {T}imoshenko beam}, Z. Angew. Math. Phys. \textbf{66} (2015),
  no.~4, 1969--1986.

\bibitem{Eck}
C.~Eck, J.~Jaru\v{s}ek, and M.~Krbec, \emph{Unilateral contact problems}, Pure
  and Applied Mathematics (Boca Raton), vol. 270, Chapman \& Hall/CRC, Boca
  Raton, FL, 2005, Variational methods and existence theorems.

\bibitem{Temam-EFN}
A.~Eden, C.~Foias, B.~Nicolaenko, and R.~Temam, \emph{Exponential attractors
  for dissipative evolution equations}, RAM: Research in Applied Mathematics,
  vol.~37, Masson, Paris; John Wiley \& Sons, Ltd., Chichester, 1994.

\bibitem{EngelNagel}
K.-J. Engel and R.~Nagel, \emph{One-parameter semigroups for linear evolution
  equations}, Graduate Texts in Mathematics, vol. 194, Springer-Verlag, New
  York, 2000, With contributions by S. Brendle, M. Campiti, T. Hahn, G.
  Metafune, G. Nickel, D. Pallara, C. Perazzoli, A. Rhandi, S. Romanelli and R.
  Schnaubelt.

\bibitem{Gatti}
S.~Gatti, M.~Grasselli, A.~Miranville, and V.~Pata, \emph{A construction of a
  robust family of exponential attractors}, Proc. Amer. Math. Soc. \textbf{134}
  (2006), no.~1, 117--127.

\bibitem{MR2271373}
S.~Gatti and V.~Pata, \emph{A one-dimensional wave equation with nonlinear
  damping}, Glasg. Math. J. \textbf{48} (2006), no.~3, 419--430.

\bibitem{Glowinski}
R.~Glowinski and P.~Le~Tallec, \emph{Augmented {L}agrangian and
  operator-splitting methods in nonlinear mechanics}, SIAM Studies in Applied
  Mathematics, vol.~9, Society for Industrial and Applied Mathematics (SIAM),
  Philadelphia, PA, 1989.

\bibitem{Hale}
J.~Hale, \emph{Asymptotic behaviour of dissipative systems}, Mathematical
  Surveys and Monographs, vol.~25, American Mathematical Society, 1988.

\bibitem{Kikuchi}
N.~Kikuchi and J.~T. Oden, \emph{Contact problems in elasticity: a study of
  variational inequalities and finite element methods}, SIAM Studies in Applied
  Mathematics, vol.~8, Society for Industrial and Applied Mathematics (SIAM),
  Philadelphia, PA, 1988.

\bibitem{KS}
K.~L. Kuttler and M.~Shillor, \emph{Vibrations of a beam between two stops},
  Dyn. Contin. Discrete Impuls. Syst. Ser. B Appl. Algorithms \textbf{8}
  (2001), no.~1, 93--110.

\bibitem{Lady}
O.~Ladyzhenskaya, \emph{Attractors for semigroups and evolution equations},
  Lezioni Lincee. [Lincei Lectures], Cambridge University Press, Cambridge,
  1991.

\bibitem{s4LC99}
Z.~Liu and S.~Zheng, \emph{Semigroups associated with dissipative systems},
  Chapman \& Hall/CRC Research Notes in Mathematics, vol. 398, Chapman \&
  Hall/CRC, Boca Raton, FL, 1999.

\bibitem{Miranville-Zelik}
A.~Miranville and S.~Zelik, \emph{Attractors for dissipative partial
  differential equations in bounded and unbounded domains}, Handbook of
  differential equations: evolutionary equations. {V}ol. {IV}, Handb. Differ.
  Equ., Elsevier/North-Holland, Amsterdam, 2008, pp.~103--200.

\bibitem{Tip1}
J.~E. Mu\~noz Rivera and A.~I. \'Avila, \emph{Rates of decay to non homogeneous
  {T}imoshenko model with tip body}, J. Differential Equations \textbf{258}
  (2015), no.~10, 3468--3490.

\bibitem{TimoPoint}
J.~E. Mu\~noz Rivera and M.~G. Naso, \emph{About the stability to {T}imoshenko
  system with pointwise dissipation}, Discrete Contin. Dyn. Syst. Ser. S
  \textbf{15} (2022), no.~8, 2289--2303.

\bibitem{RiveraHigidio}
J.~E. Mu\~noz Rivera and H.~Portillo~Oquendo, \emph{Exponential stability to a
  contact problem of partially viscoelastic materials}, J. Elasticity
  \textbf{63} (2001), no.~2, 87--111.

\bibitem{Pata-twoquestions}
V.~Pata, \emph{Two questions arising in the theory of attractors}, Evol. Equ.
  Control Theory \textbf{8} (2019), no.~3, 663--668.

\bibitem{pazy}
A.~Pazy, \emph{Semigroups of linear operators and applications to partial
  differential equations}, Applied Mathematical Sciences, vol.~44,
  Springer-Verlag, New York, 1983.

\bibitem{Teman}
R.~Teman, \emph{Infinite-dimensional dynamical systems in mechanics and
  physics}, Applied Mathematical Sciences, vol.~68, Springer Verlag, 1988.

\bibitem{Wriggers}
P.~Wriggers, \emph{Computational contact mechanics [{E}ditorial]}, Comput.
  Mech. \textbf{49} (2012), no.~6, 685.

\end{thebibliography}

\end{document}